\documentclass{amsart}
\usepackage{color,hyperref,tikz,amssymb}
\usepackage[edges]{forest}
\usepackage{shuffle}

\usetikzlibrary{matrix,arrows,shapes}
\usetikzlibrary{decorations.markings} 
\usetikzlibrary{arrows.meta}

\title{Quasisymmetric functions distinguishing trees}

\author{Jean-Christophe Aval}%
\address{LaBRI, CNRS, Universit\'e de Bordeaux, 351 cours de la Lib\'eration, 33405 Talence, France}%
 \email{jean-christophe.aval.1@u-bordeaux.fr}%

\author{Karimatou Djenabou}%
\address{African Institute for Mathematical Sciences, 6 Melrose Road, Muizenberg 7945, South Africa}%
 \email{karimatou@aims.ac.za}%

\author{Peter R. W. McNamara}%
\address{Department of Mathematics, Bucknell University, Lewisburg, PA 17837, USA}%
\email{peter.mcnamara@bucknell.edu}

{\theoremstyle{plain}
\newtheorem{theorem}{Theorem}[section]
\newtheorem{conjecture}[theorem]{Conjecture}
\newtheorem{proposition}[theorem]{Proposition}
\newtheorem{lemma}[theorem]{Lemma}
\newtheorem{corollary}[theorem]{Corollary}
\newtheorem{question}[theorem]{Question}
}
{\theoremstyle{definition}
\newtheorem{definition}[theorem]{Definition}
\newtheorem{remark}[theorem]{Remark}
\newtheorem{example}[theorem]{Example}
}

\numberwithin{equation}{section}
\numberwithin{figure}{section}

\definecolor{darkgreen}{rgb}{0,0.3,0}

\newcommand{\diG}{\overrightarrow{G}}
\newcommand{\diH}{\overrightarrow{H}}
\newcommand{\asc}{\mathrm{asc}}
\newcommand{\qsym}{{\it{QSym}}}
\newcommand{\x}{\mathbf{x}}
\newcommand{\om}{\omega}
\newcommand{\anti}{\mathrm{anti}}
\newcommand{\wt}{\mathrm{wt}}
\newcommand{\KK}[1]{\overline{K}_{#1}}
\newcommand{\kpw}{K_{(P,\om)}}
\newcommand{\des}{\mathrm{des}}
\renewcommand{\L}{\mathcal{L}}
\newcommand{\co}{\mathrm{co}}
\newcommand{\inv}[1]{\overline{#1}}

\newcommand{\Cb}{\mathcal{C}}

\newcommand{\wP}{\uparrow}
\newcommand{\sP}{\Uparrow}

\newcommand{\pa}
{\begin{tikzpicture}[scale=0.25] 
\tikzstyle{every node}=[draw, shape=circle, inner sep=1pt]; 
\draw (0,0) node (a1) {};
\draw (1,0) node (a2) {};
\draw (0,1) node (a3) {};
\draw (1,1) node (a4) {};
\draw (a1) -- (a3) -- (a2) -- (a4) ;
\end{tikzpicture}}

\newcommand{\pc}
{\begin{tikzpicture}[scale=0.25] 
\tikzstyle{every node}=[draw, shape=circle, inner sep=1pt]; 
\draw (0,0) node (a1) {};
\draw (1,0) node (a2) {};
\draw (0,1) node (a3) {};
\draw (1,1) node (a4) {};
\draw (a1) -- (a3) -- (a2) -- (a4) -- (a1) ;
\end{tikzpicture}}

\newcommand{\pf}
{\begin{tikzpicture}[scale=0.25] 
\tikzstyle{every node}=[draw, shape=circle, inner sep=1pt]; 
\draw (0,1) node (a1) {};
\draw (-0.5,0) node (a2) {};
\draw (0.5,0) node (a3) {};
\draw (a2) -- (a1) ;
\draw [double distance = 1pt] (a3) -- (a1) ;
\end{tikzpicture}}

\newcommand{\pe}
{\begin{tikzpicture}[scale=0.25] 
\tikzstyle{every node}=[draw, shape=circle, inner sep=1pt]; 
\draw (0,0) node (a1) {};
\draw (-0.5,1) node (a2) {};
\draw (0.5,1) node (a3) {};
\draw (a2) -- (a1) ;
\draw [double distance = 1pt] (a3) -- (a1) ;
\end{tikzpicture}}

\newcommand{\shiftpics}[1]{\raisebox{-0.6ex}{#1}}

\newcommand{\pac}{(\shiftpics{\pa}\,,\shiftpics{\pc})}
\newcommand{\pef}{[\shiftpics{\pe},\shiftpics{\pf}]}

\begin{document}

\begin{abstract}
A famous conjecture of Stanley states that his chromatic symmetric function distinguishes trees.  As a quasisymmetric analogue, we conjecture that the chromatic quasisymmetric function of Shareshian and Wachs and of Ellzey distinguishes directed trees.  This latter conjecture would be implied by an affirmative answer to a question of Hasebe and Tsujie about the $P$-partition enumerator distinguishing posets whose Hasse diagrams are trees.   They proved the case of rooted trees and our results include a generalization of their result.
\end{abstract}

\keywords{chromatic, quasisymmetric function, digraph, poset, P-partition,  rooted tree}

\maketitle


\section{Introduction}
 
As an extension of the chromatic polynomial $\chi_G(k)$ of a graph $G =  (V,E)$, Stanley \cite{Sta95} introduced the chromatic symmetric function $X_G(\x)$ defined by
\begin{equation}\label{equ:stanley}
X_G(\x) = \sum_\kappa x_1^{\# \kappa^{-1}(1)} x_2^{\# \kappa^{-1}(2)} \cdots
\end{equation}
where the sum is over all proper colorings $\kappa : V \to \{1,2,\ldots\}$.  Observe that setting $x_i=1$ for $1\leq i \leq k$ and $x_i=0$ otherwise yields $\chi_G(k)$.  Two famous and unsolved conjectures appear in \cite{Sta95}.  One of these, known as the Stanley--Stembridge conjecture \cite[Conj.~5.5]{StSt93}\cite[Conj.~5.1]{Sta95} is about the $e$-positivity of $X_G(\x)$ for incomparability graphs of $(\mathbf{3}+\mathbf{1})$-free posets and does not concern us here.  Of more interest to us is that Stanley gave a pair of non-isomorphic graphs on five vertices (see Figure~\ref{fig:equalxg}(c) with the arrows removed) that have the same $X_G(\x)$.  As a result, we say that $X_G(\x)$ does not \emph{distinguish} graphs.  Stanley stated ``We do not know whether $X_G$ distinguishes trees.''  Subsequent papers, such as \cite{AMZ17,AMZ21,AlZa14,Fou03,HeJi19,HuCh20,LoSe19,MMW08,OrSc14,SST15}, have established that $X_G(\x)$ distinguishing trees is certainly worthy of being called a conjecture; for example, \cite{HeJi19} shows that $X_G(\x)$ distinguishes trees with up to 29 vertices.

We focus on a generalization of $X_G(\x)$ to labeled graphs introduced by Shareshian and Wachs \cite{ShWa16}, denoted $X_G(\x, t)$, which has an extra parameter $t$ and is now just a quasisymmetric function in general.  In fact, we will use a further generalization of $X_G(\x)$ to directed graphs (digraphs) $\diG$ due to Ellzey \cite{Ell17,Ell17+,EllThesis} and denoted $X_{\diG}(\x,t)$. 

Our original goal in this project was to study equality among $X_{\diG}(\x, t)$.  It is not obvious if $X_{\diG}(\x, t)$ will be more or less successful at distinguishing digraphs compared to $X_G(\x)$ distinguishing graphs: there are far more digraphs than graphs for a given number of vertices, but $X_{\diG}(\x, t)$ contains more information than $X_G(\x)$.  It is not hard to find digraphs with the same $X_{\diG}(\x, t)$; three such equalities are given in Figure~\ref{fig:equalxg}.  

\tikzset{->-/.style={decoration={
  markings,
  mark=at position #1 with {\arrow[scale=1.5]{Straight Barb}}},postaction={decorate}}}
  
\begin{figure}
\begin{center}
\begin{tikzpicture}[scale=1.5] 

\begin{scope} 
\tikzstyle{every node}=[draw, shape=circle, inner sep=2pt]; 
\begin{scope}  
\draw (0,0) node (a1) {};
\draw (1,0) node (a2) {};
\draw (0,1) node (a3) {};
\draw (1,1) node (a4) {};
\draw[->-=0.55] (a1) -- (a3);
\draw[->-=0.55]  (a3) -- (a2);
\draw[->-=0.55]  (a3) -- (a4);
\draw[->-=0.55]  (a4) -- (a2);
\end{scope}  
\begin{scope}[yshift=10ex] \draw (0,0) node (a1) {};
\draw (1,0) node (a2) {};
\draw (0,1) node (a3) {};
\draw (1,1) node (a4) {};
\draw[->-=0.55] (a3) -- (a1);
\draw[->-=0.55] (a2) -- (a3);
\draw[->-=0.55] (a4) -- (a3);
\draw[->-=0.55] (a2) -- (a4);
\end{scope}
\end{scope}
\begin{scope}
\draw (0.5,-0.75) node {(a)};
\end{scope} 

\begin{scope}[xshift=15ex]
\begin{scope} 
\tikzstyle{every node}=[draw, shape=circle, inner sep=2pt]; 
\begin{scope}  
\draw (0,0) node (a1) {};
\draw (1,0) node (a2) {};
\draw (0,1) node (a3) {};
\draw (1,1) node (a4) {};
\draw[->-=0.55] (a1) -- (a3);
\draw[->-=0.55] (a3) -- (a4);
\draw[->-=0.55] (a2) -- (a1);
\draw[->-=0.75] (a4) -- (a1);
\draw[->-=0.35] (a3) -- (a2);
\draw[->-=0.55] (a4) to [bend left=20] (a2);
\draw[->-=0.55] (a2) to [bend left=20] (a4);
\end{scope}  
\begin{scope}[yshift=10ex] 
\draw (0,0) node (a1) {};
\draw (1,0) node (a2) {};
\draw (0,1) node (a3) {};
\draw (1,1) node (a4) {};
\draw[->-=0.55] (a1) -- (a3);
\draw[->-=0.55] (a3) -- (a4);
\draw[->-=0.55] (a1) -- (a2);
\draw[->-=0.75] (a4) -- (a1);
\draw[->-=0.75] (a2) -- (a3);
\draw[->-=0.55] (a4) to [bend left=20] (a2);
\draw[->-=0.55] (a2) to [bend left=20] (a4);
\end{scope}
\end{scope} 
\begin{scope}
\draw (0.5,-0.75) node {(b)};
\end{scope}
\end{scope} 

\begin{scope}[xshift=30ex] 
\begin{scope} 
\tikzstyle{every node}=[draw, shape=circle, inner sep=2pt]; 
\begin{scope}  
\draw (0,0) node (a1) {};
\draw (1,0) node (a2) {};
\draw (0,1) node (a3) {};
\draw (1,1) node (a4) {};
\draw (2,0) node (a5) {};
\draw[->-=0.55] (a1) -- (a3);
\draw[->-=0.55] (a1) -- (a2);
\draw[->-=0.55] (a2) -- (a3);
\draw[->-=0.55] (a2) -- (a4);
\draw[->-=0.55] (a3) -- (a4);
\draw[->-=0.55] (a4) -- (a5);
\end{scope}
\begin{scope}[yshift=10ex] 
\draw (0,0) node (a1) {};
\draw (0,1) node (a2) {};
\draw (1,0.5) node (a3) {};
\draw (2,0) node (a4) {};
\draw (2,1) node (a5) {};
\draw[->-=0.55] (a1) -- (a2);
\draw[->-=0.55] (a1) -- (a3);
\draw[->-=0.55] (a2) -- (a3);
\draw[->-=0.55] (a3) -- (a4);
\draw[->-=0.55] (a3) -- (a5);
\draw[->-=0.55] (a4) -- (a5);
\end{scope}
\end{scope}
\begin{scope}
\draw (1,-0.75) node {(c)};
\end{scope}
\end{scope}

\end{tikzpicture}
\end{center}
\caption{Pairs of digraphs with equal chromatic quasisymmetric functions}
\label{fig:equalxg}
\end{figure}
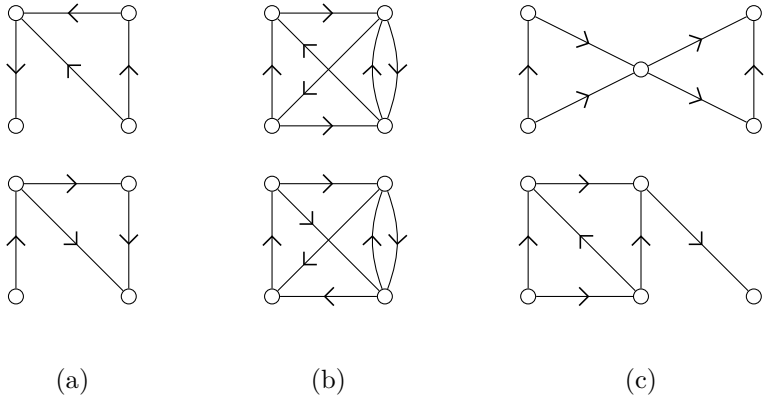

One way to bring Stanley's conjecture into the quasisymmetric setting would be by stating that $X_{\diG}(\x, t)$ distinguishes directed acyclic graphs, but (a) and (c) of Figure~\ref{fig:equalxg} show that such a statement is false.  Instead, we offer the following conjecture, which is a natural extension of Stanley's conjecture to the quasisymmetric setting; it is stated as a question in \cite{AlSu21}.

\begin{conjecture}\label{con:chromatic_trees}
$X_{\diG}(\x, t)$ distinguishes directed trees.
\end{conjecture}

In other words, the conjecture states that if directed trees
$\diG$ and $\diH$ are not isomorphic then $X_{\diG}(\x, t) \neq X_{\diH}(\x, t)$.  Our approach to tackling Conjecture~\ref{con:chromatic_trees} will be to translate it into a question about posets.  When $\diG$ is a directed acyclic graph, we can represent $\diG$ as a poset  by saying $v_i \leq v_j$ if there is a directed path from $v_i$ to $v_j$.  In other words, draw $\diG$ so that all the edges point upwards on the page; removing the arrows results in a Hasse diagram of a poset, possibly with some redundant edges; see Figure~\ref{fig:dg2poset}.  Let $P$ denote the resulting poset.  As we will see, the coefficient of the highest power of $t$ that appears in $X_{\diG}(\x, t)$ will be the well-known strict $P$-partition enumerator $\KK{P}(\x)$, a fact previously mentioned in \cite[Theorem~7.4]{AlSu21} and \cite[p.~11]{Ell17+}.   Obviously two chromatic quasisymmetric functions are different if their coefficients on the highest power of $t$ are different.  So to prove Conjecture~\ref{con:chromatic_trees} it suffices to prove the following conjecture.  A poset being a tree simply means its Hasse diagram is a tree.

\begin{conjecture}\label{con:poset_trees}
$\KK{P}(\x)$ distinguishes posets that are trees. 
\end{conjecture}

We have verified this conjecture for all such posets with at most 11 elements.  A main advantage of the poset setting is that Conjecture~\ref{con:poset_trees}, and equality among $\KK{P}$ in general, has already been studied in the literature as we detail in Section~\ref{sec:poset_results}.  In fact, Conjecture~\ref{con:poset_trees} appears as a question in \cite[Problem~6.1]{HaTs17}.  Section~\ref{sec:prelims} will give the necessary background, including definitions of the quasisymmetric functions mentioned above.  While we only needed the  \emph{strict} $P$-partition enumerator above, in Section~\ref{sec:strict_weak_edges}, we consider the original $(P,\om)$-partition enumerator $\kpw(\x)$ which allows for a mixture of strict and weak relations.  Such $(P,\om)$ are called $\emph{labeled posets}$.  Interestingly, $\kpw$ does \emph{not} distinguish labeled posets that are trees, but we offer the following conjecture.  A poset that is a tree is said to be a \emph{rooted tree} if it has a unique minimal element.

\begin{conjecture}\label{con:labeled_rooted_trees}
$\kpw(\x)$ distinguishes labeled posets that are rooted trees.
\end{conjecture}
 
 Hasebe and Tsujie \cite{HaTs17} have shown the case when all the relations are weak (or all strict), and we generalize their result by establishing Conjecture~\ref{con:labeled_rooted_trees} for a class of labeled rooted trees that we call \emph{fair trees}.  The class of fair trees interpolates between rooted trees with all weak edges and those with all strict edges.  We believe our result on fair trees (Theorem~\ref{thm:fair-trees}) is the first non-trivial result stating that $\kpw(\x)$ distinguishes a class of posets with a mixture of strict and weak edges.  We conclude in Section~\ref{sec:conclusion} with several directions for further study.

\section{Preliminaries}\label{sec:prelims}

\subsection{The chromatic symmetric function for directed graphs}

Let $\mathbb{P}$ denote the set of positive integers.  For a positive integer $n$, we write $[n]$ to denote the set $\{1,\ldots,n\}$.

\begin{definition}{\cite{Ell17+}}\label{def:xg}
Let $\diG = (V,E)$ be a directed graph.  Given a proper coloring $\kappa: V \to \mathbb{P}$ of $\diG$, we define an \emph{ascent} of $\kappa$ to be a directed edge $(v_i, v_j) \in E$ with $\kappa(v_i) < \kappa(v_j)$, and we let $\asc(\kappa)$ denote the number of ascents of $\kappa$.  The chromatic quasisymmetric function of $\diG$ is 
\begin{equation}\label{equ:xgt}
X_{\diG}(\x, t) = \sum_{\kappa} t^{\asc(\kappa)} x_1^{\# \kappa^{-1}(1)} x_2^{\# \kappa^{-1}(2)} \cdots
\end{equation}
where the sum is over all proper colorings of $\diG$.
\end{definition}

Setting $t=1$ yields Stanley's chromatic symmetric function $X_G(\x)$ of~\eqref{equ:stanley}, where $G$ denotes the undirected version of $\diG$.  When $\diG$ is a directed acyclic graph, which is the case of most interest to us, $X_{\diG}(\x, t)$ coincides with the chromatic quasisymmetric function of Shareshian and Wachs.  We use the digraph setting because in the labeled graph setting of Shareshian--Wachs, there are lots of trivial equalities among $X_G(\x, t)$ that result just from relabeling.

\begin{example}
Let $\diG$ be the 3-element path with the directions as shown below.  With colors $a<b<c$, the proper colorings $\kappa$ of  $\diG$ fall into the 8 classes given by the following table.
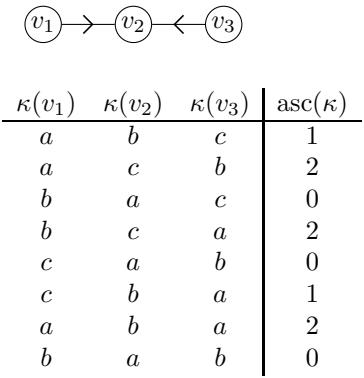
\[
\begin{tikzpicture}[scale=1.2] 
\tikzstyle{every node}=[draw, shape=circle, inner sep=1pt]; 
\draw (0,0) node (v1) {$v_1$};
\draw (1,0) node (v2) {$v_2$};
\draw (2,0) node (v3) {$v_3$};
\draw[->-=0.6] (v1) -- (v2);
\draw[->-=0.6] (v3) -- (v2);
\end{tikzpicture} 
\ \ \ \ \ \ \ \ \ \ \ 
\]
\smallskip
\[
\begin{array}{ccc|c}
\kappa(v_1) & \kappa(v_2) & \kappa(v_3) & \asc(\kappa) \\ \hline
a & b & c & 1\\
a & c & b & 2\\
b & a & c & 0\\
b & c & a & 2\\
c & a & b  & 0\\
c & b & a & 1\\
a & b & a & 2\\
b & a & b & 0
\end{array}
\]
Thus 
\[
X_{\diG}(\x,t) = (2 + 2t + 2t^2)M_{111} + t^2 M_{21} + M_{12}\,,
\]
where $M$ denotes the basis of monomial quasisymmetric functions (see Subsection~\ref{sub:qsym} for the necessary background on quasisymmetric functions).
Setting $t=1$ gives $X_G(\x) = 6m_{111} + m_{21}$ from which we get $\chi_G(k) =  6 \binom{k}{3} +k(k-1) = k(k-1)^2$, as expected.
\end{example}

Let us make a couple of observations about types of $X_{\diG}(\x,t)$-equality that arise.  By setting $t=1$, we know equal $X_{\diG}(\x,t)$ means the underlying undirected graphs must have equal $X_G(\x)$ and the examples in Figure~\ref{fig:equalxg} show two scenarios: either the underlying undirected graphs are isomorphic, or they are not isomorphic but have equal $X_G(\x)$.  The $X_G(\x)$-equality implied by Figure~\ref{fig:equalxg}(c) is the one given by Stanley in \cite{Sta95}. 

Figure~\ref{fig:equalxg}(a) shows an example of $X_{\diG}(\x, t)$ being invariant under reversal of all the edge directions.  Letting $\alpha^{\mathrm{rev}}$ denote the reversal of the composition $\alpha$, this invariance will hold whenever the coefficients of $M_\alpha$ and $M_{\alpha ^{\mathrm{rev}}}$ in $X_{\diG}(\x, t)$ are equal for all $\alpha$, so in particular when $X_{\diG}(\x, t)$ is symmetric \cite[Cor.~2.7]{ShWa16} \cite[Prop.~2.6]{Ell17+}. But not all equalities among $X_{\diG}(\x, t)$ with isomorphic underlying graphs arise from reversal of all edges, as shown by Figure~\ref{fig:equalxg}(b).

\subsection{The poset perspective}

As mentioned in the Introduction, when $\diG$ is a directed acyclic graph we can view it as a poset; see Figure~\ref{fig:dg2poset} for an example, with a coloring given by numbers next to each vertex. 

\begin{figure}
\begin{center}
\begin{tikzpicture}[scale=1.0]  
\begin{scope}
\begin{scope}
\tikzstyle{every node}=[draw, shape=circle, inner sep=2pt];
\draw (0,0) node (e) {$v_4$};
\draw (-1,1) node (b) {$v_1$};
\draw (1,1) node (d) {$v_5$};
\draw (0,2) node (c) {$v_2$};
\draw (-2.4,1) node (a) {$v_3$};
\draw[->-=0.55] (a) -- (b);
\draw[->-=0.55] (c) -- (b);
\draw[->-=0.55] (e) -- (b);
\draw[->-=0.55] (d) -- (e);
\draw[->-=0.55] (d) -- (c);
\draw[->-=0.55] (d) -- (b);
\end{scope}
\draw (0.45, 0) node {\textcolor{blue}{\textbf{2}}};
\draw (-2.4, 1.5) node {\textcolor{blue}{\textbf{1}}};
\draw (-1, 1.5) node {\textcolor{blue}{\textbf{5}}};
\draw (0.45, 2) node {\textcolor{blue}{\textbf{3}}};
\draw (1, 1.5) node {\textcolor{blue}{\textbf{1}}};
\draw (0,-1) node {(a)};
\end{scope}
\begin{scope}[xshift=10em]
\begin{scope}
\tikzstyle{every node}=[draw, shape=circle, inner sep=2pt];
\draw (0,0) node (d) {$v_5$};
\draw (-1,1) node (e) {$v_4$};
\draw (1,1) node (c) {$v_2$};
\draw (0,2) node (b) {$v_1$};
\draw (2,1) node (a) {$v_3$};
\draw (d) -- (e) -- (b) -- (a);
\draw (d) -- (c) -- (b);
\end{scope}
\draw (0.45, 0) node {\textcolor{blue}{\textbf{1}}};
\draw (2, 1.5) node {\textcolor{blue}{\textbf{1}}};
\draw (-1, 1.5) node {\textcolor{blue}{\textbf{2}}};
\draw (-0.45, 2) node {\textcolor{blue}{\textbf{5}}};
\draw (1, 0.52) node {\textcolor{blue}{\textbf{3}}};
\draw (0,-1) node {(b)};
\end{scope}
\end{tikzpicture}
\end{center}
\caption{Converting from a proper coloring of a digraph to a strict $P$-partition.  The numbers next to each node correspond to a coloring in the digraph on the left and the corresponding $(P,\om)$-partition of the labeled poset on the right.}
\label{fig:dg2poset}
\end{figure}
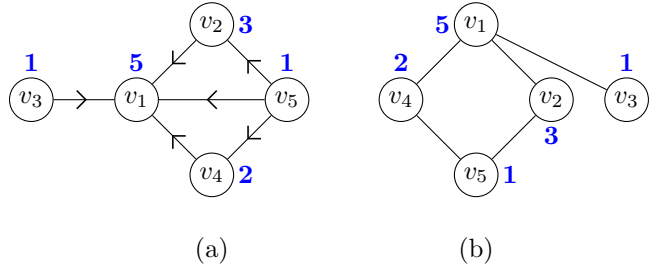

Now consider the coefficient of the highest power of $t$ that appears in $X_{\diG}(\x, t)$.  This coefficient enumerates colorings that strictly increase along every directed edge, as in Figure~\ref{fig:dg2poset}(a).  We now compare this to the definition of Stanley's $(P,\om)$-partitions.

Let $P$ be a poset with $n$ elements; we write $|P|=n$.  Denote the order relation on $P$ by $\leq_P$, while $\leq$ denotes the usual order on the positive integers.  A \emph{labeling} of $P$ is a bijection $\om : P \to [n]$.  A \emph{labeled poset} $(P,\om)$ is then a poset $P$ with an associated labeling $\om$.  

\begin{definition}\label{def:popartition}
For a labeled poset $(P,\om)$, a \emph{$(P,\om)$-partition} is a map $f$ from $P$ to the positive integers satisfying the following two conditions:
\begin{itemize}
\item if $a <_P b$, then $f(a) \leq f(b)$, i.e., $f$ is order-preserving;
\item if $a <_P b$ and $\om(a) > \om(b)$, then $f(a) < f(b)$.
\end{itemize}
\end{definition}
In other words, a $(P,\om)$-partition is an order-preserving map from $P$ to the positive integers with certain strictness conditions determined by $\om$.  Examples of $(P,\om)$-partitions $f$ are given in Figure~\ref{fig:popartitions}, where the images under $f$ are written in bold and blue next to the nodes.  

The meaning of the double edges in the figure follows from the following observation about Definition~\ref{def:popartition}.  For $a, b \in P$, we say that $a$ is \emph{covered} by $b$ in $P$, denoted $a \prec_P b$, if $a <_P b$ and there does not exist $c$ in $P$ such that $a <_P c <_P b$.  Note that a definition equivalent to Definition~\ref{def:popartition} is obtained by replacing both appearances of the relation $a <_P b$ with the relation $a \prec_P b$.  In other words, we require that $f$ be order-preserving along the edges of the Hasse diagram of $P$, with $f(a) < f(b)$ when the edge $a \prec_P b$ satisfies $\om(a) > \om(b)$.  With this in mind, we will consider those edges $a \prec_P b$ with $\om(a) > \om(b)$ as \emph{strict edges} and we will represent them in Hasse diagrams by double lines.  Similarly, edges $a \prec_P b$ with $\om(a) < \om(b)$ will be called \emph{weak edges} and will be represented by single lines.  

From the point-of-view of $(P,\om)$-partitions, the labeling $\om$ only determines which edges are strict and which are weak.  Therefore, we say that two labeled posets $(P,\om)$ and $(Q,\om')$ are \emph{isomorphic} if $P$ and $Q$ are isomorphic as posets and they have equivalent sets of strict and weak edges according to a poset isomorphism.  Thus many of our figures from this point on will not show the labeling $\om$, but instead show some collection of strict and weak edges determined by an underlying $\om$.  
\begin{center}
\begin{figure}[htbp]
\begin{tikzpicture}[scale=0.7]
\begin{scope}
\begin{scope}
\tikzstyle{every node}=[shape=circle, inner sep=2pt]; 
\draw (0,0) node[draw] (a1) {1} +(0.6,-0.3) node {\textcolor{blue}{\textbf{4}}};
\draw (0,1.5) node[draw] (a4) {4} +(0.0,0.65) node {\textcolor{blue}{\textbf{3}}};
\draw (-1.5,2.5) node[draw] (a2) {2} +(-0.6,-0.3) node {\textcolor{blue}{\textbf{4}}};
\draw (1.5,2.5) node[draw] (a3) {3} +(0.6,-0.3) node {\textcolor{blue}{\textbf{7}}};
\draw (a2) -- (a1) -- (a3);
\draw[double distance=2pt] (a3) --(a4) -- (a2);
\end{scope}
\draw (0,-1.3) node {(a)};
\end{scope}

\begin{scope}[xshift=15em]
\begin{scope}
\tikzstyle{every node}=[shape=circle, inner sep=2pt]; 
\draw (0,0) node[draw] (a1) {1} +(0.6,-0.3) node {\textcolor{blue}{\textbf{4}}};
\draw (0,1.5) node[draw] (a4) {2} +(0.0,0.65) node {\textcolor{blue}{\textbf{3}}};
\draw (-1.5,2.5) node[draw] (a2) {3} +(-0.6,-0.3) node {\textcolor{blue}{\textbf{4}}};
\draw (1.5,2.5) node[draw] (a3) {4} +(0.6,-0.3) node {\textcolor{blue}{\textbf{4}}};
\draw (a2) -- (a1) -- (a3);
\draw(a3) --(a4) -- (a2);
\end{scope}
\draw (0,-1.3) node {(b)};
\end{scope}

\begin{scope}[xshift=30em]
\begin{scope}
\tikzstyle{every node}=[shape=circle, inner sep=2pt]; 
\draw (0,0) node[draw] (a1) {4} +(0.6,-0.3) node {\textcolor{blue}{\textbf{1}}};
\draw (0,1.5) node[draw] (a4) {3} +(0.0,0.65) node {\textcolor{blue}{\textbf{3}}};
\draw (-1.5,2.5) node[draw] (a2) {1} +(-0.6,-0.3) node {\textcolor{blue}{\textbf{4}}};
\draw (1.5,2.5) node[draw] (a3) {2} +(0.6,-0.3) node {\textcolor{blue}{\textbf{7}}};
\draw[double distance=2pt]  (a2) -- (a1) -- (a3);
\draw[double distance=2pt] (a3) --(a4) -- (a2);
\end{scope}
\draw (0,-1.3) node {(c)};
\end{scope}
\end{tikzpicture}
\caption{Examples of $(P,\om)$-partitions}
\label{fig:popartitions}
\end{figure}
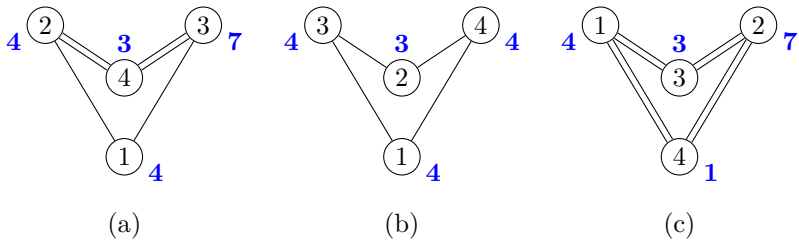
\end{center}

If $\om$ is order-preserving, as in Figure~\ref{fig:popartitions}(b), then $P$ is said to be \emph{naturally labeled} and all the edges are weak.  In this case, we typically omit reference to the labeling and so a $(P,\om)$-partition is then traditionally called a $P$-partition.  At the other extreme, when $\om$ is order-reversing, as in Figure~\ref{fig:popartitions}(c), $f$ must strictly increase along each edge. Such $(P,\om)$-partitions that are required to strictly increase along each edge will be called \emph{strict $P$-partitions}.  

Returning to Figure~\ref{fig:dg2poset}, the key observation is now clear: the proper colorings of a directed acyclic graph $\diG$ that contribute to the coefficient of the highest power of $t$ in $X_{\diG}(\x,t)$ are in bijection with strict $P$-partitions of the corresponding poset $P$.  

To make the algebraic connection, the well-known $(P,\om)$-partition enumerator is defined by 
\begin{equation}\label{equ:kp}
\kpw(\x) = \sum_{f} x_1^{\#f^{-1}(1)} x_2^{\# f^{-1}(2)} \cdots
\end{equation}
where the sum is over all $(P,\om)$-partitions $f:P \to \mathbb{P}$.  When all the edges of $(P,\om)$ are weak, so the sum is over $P$-partitions, we will denote the $P$-partition enumerator $\kpw(\x)$ simply by $K_P(\x)$ or just $K_P$.  Similarly, we will use $\KK{P}$ to denote $\kpw(\x)$ when all the edges are strict, thus enumerating strict $P$-partitions.  Comparing~\eqref{equ:xgt} and~\eqref{equ:kp} when $P$ is the poset corresponding to a directed acyclic graph $\diG$, we see that $\KK{P}(\x)$  is exactly the coefficient of the highest power of $t$ in $X_{\diG}(\x,t)$.   This connection between $\KK{P}$ and $X_{\diG}(\x, t)$ has previously been mentioned in \cite[Theorem~7.4]{AlSu21} \cite[p.~11]{Ell17+}.  As a corollary, if Conjecture~\ref{con:poset_trees} is true, then so is Conjecture~\ref{con:chromatic_trees}.

With this implication now established, we will work almost entirely in the poset setting.  Although the direct connection between $X_{\diG}(\x, t)$ and $\kpw$ uses the setting of strict $P$-partitions and $\KK{P}$, most of the results in the literature work with $P$-partitions and $K_P$.  However, for equality questions, the two settings are equivalent since $\KK{P}$ can be obtained from $K_P$ and vice-versa; see \cite[\S 3]{McWa14} for the full details of this equivalence and involutions on $(P,\om)$-partition enumerators.  In particular, Conjecture~\ref{con:poset_trees} can be restated as the assertion that $K_P(\x)$ distinguishes posets that are trees.

\begin{remark}
In addition to its simple statement, Conjecture~\ref{con:poset_trees} has the virtue that some natural more general statements are false.  The first example of non-isomorphic posets with the same $K_P$ was given in \cite{McWa14} and appears in Figure~\ref{fig:equal_kp}(a).  A $\emph{bowtie}$ is the poset consisting of elements $a_1,a_2,b_1,b_2$ with cover relations $a_i < b_j$ for all $i,j$.  Notice that each poset in Figure~\ref{fig:equal_kp}(a) has a bowtie as an induced subposet.  Otherwise, we say the poset is \emph{bowtie-free}.  Weakening the tree hypothesis of Conjecture~\ref{con:poset_trees} to bowtie-free results in a false statement, with Figure~\ref{fig:equal_kp}(b) being the smallest counterexample.
\begin{figure}
\begin{center}
\begin{tikzpicture}[scale=1.0] 

\begin{scope}
\tikzstyle{every node}=[draw, shape=circle, inner sep=2pt]; 
\begin{scope}
\draw (0,0) node (a1) {};
\draw (1,0) node (a2) {};
\draw (0,1) node (a3) {};
\draw (1,1) node (a4) {};
\draw (-1,1) node (a5) {};
\draw (0,2) node (a6) {};
\draw (1,2) node (a7) {};
\draw (a1) -- (a5) -- (a7) -- (a4) -- (a2) -- (a6) -- (a3) -- (a1);
\end{scope}
\begin{scope}[yshift=18ex]
\draw (0,0) node (a1) {};
\draw (1,0) node (a2) {};
\draw (0,1) node (a3) {};
\draw (1,1) node (a4) {};
\draw (-1,1) node (a5) {};
\draw (0,2) node (a6) {};
\draw (1,2) node (a7) {};
\draw (a1) -- (a5) -- (a6) -- (a3) -- (a1) -- (a7) -- (a4) -- (a2) -- (a6);
\end{scope}
\end{scope}
\begin{scope}
\draw (0.5,-0.75) node {(a)};
\end{scope} 

\begin{scope}[xshift=25ex]
\begin{scope}
\tikzstyle{every node}=[draw, shape=circle, inner sep=2pt]; 
\begin{scope}
\draw (0,0) node (a1) {};
\draw (1,0) node (a2) {};
\draw (0,1) node (a3) {};
\draw (1,1) node (a4) {};
\draw (-1,1) node (a5) {};
\draw (0,2) node (a6) {};
\draw (1,2) node (a7) {};
\draw (2,0) node (a8) {};
\draw (2,1) node (a9) {};
\draw (2,2) node (a10) {};
\draw (a1) -- (a5) -- (a10) -- (a9) -- (a8) -- (a7) -- (a4) -- (a2) -- (a6) -- (a3) -- (a1);
\end{scope}
\begin{scope}[yshift=18ex] 
\draw (0,0) node (a1) {};
\draw (1,0) node (a2) {};
\draw (0,1) node (a3) {};
\draw (-1,1) node (a5) {};
\draw (0,2) node (a6) {};
\draw (1,2) node (a7) {};
\draw (2,0.6) node (a8) {};
\draw (2,1.4) node (a9) {};
\draw (3,0) node (a10) {};
\draw (3,2) node (a11) {};
\draw (a1) -- (a5) -- (a6) -- (a2) -- (a9) -- (a11) -- (a10) -- (a8) -- (a7) -- (a1);
\draw (a1) -- (a3) -- (a6);
\end{scope}
\end{scope}
\begin{scope}
\draw (1,-0.75) node {(b)};
\end{scope}
\end{scope}
\end{tikzpicture}
\end{center}
\caption{Pairs of posets with equal $P$-partition enumerators}
\label{fig:equal_kp}
\end{figure}
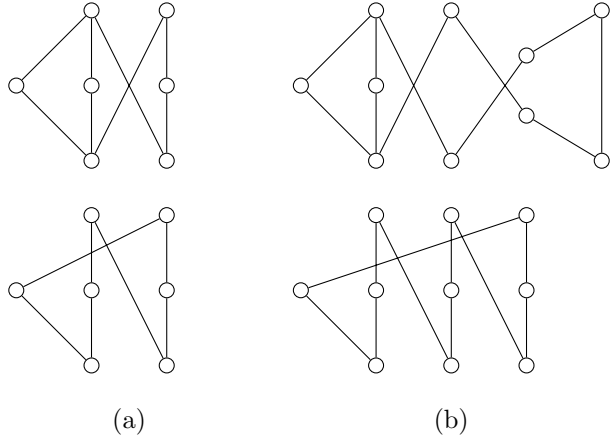

Conjecture~\ref{con:poset_trees} equivalently states that for posets $P$ and $Q$ that are trees, if $\KK{P} = \KK{Q}$ then $P$ and $Q$ are isomorphic.  We can consider weakenings of the equality hypothesis in this statement.  The \emph{$F$-support} of a quasisymmetric function $f$ is the set of compositions $\alpha$ for which the coefficient of $F_\alpha$ is non-zero when $f$ is expanded in the $F$-basis.  But the $F$-support of $\KK{P}$ and  $\KK{Q}$ being equal does not imply that $P$ and $Q$ are isomorphic, with a counterexample given by the posets in Figure~\ref{fig:f_support}, which both have $F$-support $\{221, 212, 122, 2111, 1211, 1121, 1112, 11111\}$.

In contrast, see Subsection~\ref{sub:ps} for versions of Conjecture~\ref{con:poset_trees} that use much less information than in $\KK{P}(\x)$ to distinguish posets that are trees.

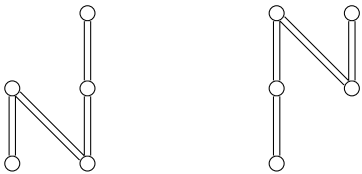
\begin{figure}
\begin{center}
\begin{tikzpicture}[scale=1.0] 
\tikzstyle{every node}=[draw, shape=circle, inner sep=2pt]; 
\begin{scope}  
\draw (1,0) node (a1) {};
\draw (0,0) node (a2) {};
\draw (1,1) node (a3) {};
\draw (0,1) node (a4) {};
\draw (1,2) node (a5) {};
\draw[double distance=2pt] (a2) -- (a4) -- (a1) -- (a3) -- (a5);
\end{scope}  
\begin{scope}[xshift=10em] 
\draw (0,0) node (a1) {};
\draw (0,1) node (a2) {};
\draw (1,1) node (a3) {};
\draw (0,2) node (a4) {};
\draw (1,2) node (a5) {};
\draw[double distance=2pt] (a1) -- (a2) -- (a4) -- (a3) -- (a5);
\end{scope}
\end{tikzpicture}
\end{center}
\caption[caption]{The strict $P$-partition enumerators of these two posets have the same $F$-support}
\label{fig:f_support}
\end{figure}
\end{remark}

\subsection{Quasisymmetric functions}\label{sub:qsym}

It follows directly from its definition that $\kpw$ is a quasisymmetric function.  In fact, $\kpw$ served as a motivating example for Gessel's original definition \cite{Ges84} of quasisymmetric functions.   

For our purposes, quasisymmetric functions are elements of $\mathbb{Z}[[x_1, x_2, \ldots]]$ and we denote the ring of quasisymmetric functions by \qsym. We will make use of both of the classical bases for \qsym.  If $\alpha = (\alpha_1, \alpha_2, \ldots, \alpha_k)$ is a composition of $n$, then we define the \emph{monomial quasisymmetric function} $M_\alpha$ by
\[
M_\alpha = \sum_{i_1 < i_2 < \ldots < i_k} x_{i_1}^{\alpha_1} x_{i_2}^{\alpha_2} \cdots x_{i_k}^{\alpha_k}.
\]

As we know, compositions $\alpha = (\alpha_1, \alpha_2, \ldots, \alpha_k)$ of $n$ are in bijection with subsets of $[n-1]$, and let $S(\alpha)$ denote the set
$\{\alpha_1, \alpha_1+\alpha_2, \ldots, \alpha_1+\alpha_2 + \cdots+ \alpha_{k-1}\}$.  Thus we also denote $M_\alpha$ by $M_{S(\alpha), n}$.  Notice that these two notations are distinguished by the latter one including the subscript $n$; this subscript is helpful since $S(\alpha)$ does not uniquely determine $n$.  

The second classical basis is composed of the \emph{fundamental quasisymmetric functions} $F_\alpha$ defined by 
\begin{equation}\label{equ:F2M}
F_\alpha = F_{S(\alpha), n} = \sum_{S(\alpha) \subseteq T  \subseteq [n-1]} M_{T,n}\,.
\end{equation}
The relevance of this latter basis to $\kpw$ is due to Theorem~\ref{theo:KinF} below, which first appeared in \cite{StaThesis71,StaThesis} and, in the language of quasisymmetric functions, in \cite{Ges84}.  

Every permutation $\pi \in S_n$ has a descent set $\des(\pi)$ given by $\{i \in [n-1] : \pi(i) > \pi(i+1)\}$, and we will call the corresponding composition of $n$ the \emph{descent composition} of $\pi$, denoted $\co(\pi)$.  For example, if $\pi = 243561$, then $\des(\pi) = \{2,5\}$ and $\co(\pi) = 231$.  Let $\L(P,\om)$ denote the set of all linear extensions of $P$, regarded as permutations of the $\om$-labels of $P$.  For example, for the labeled poset in Figure~\ref{fig:popartitions}(a), $\L(P,\om) = \{1423, 1432, 4123, 4132\}$.  

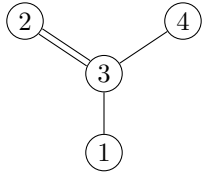
\begin{figure}
\begin{center}
\begin{tikzpicture}[scale=0.7]
\tikzstyle{every node}=[shape=circle, inner sep=2pt]; 
\draw (0,0) node[draw] (a1) {1};
\draw (0,1.5) node[draw] (a3) {3};
\draw (-1.5,2.5) node[draw] (a2) {2};
\draw (1.5,2.5) node[draw] (a4) {4};
\draw (a1) -- (a3) -- (a4);
\draw[double distance=2pt] (a3) -- (a2);
\end{tikzpicture}
\caption{The labeled poset of Example~\ref{exa:kexpansion}}
\label{fig:labeled_poset}
\end{center}
\end{figure}

\begin{theorem}[\cite{Ges84,StaThesis71,StaThesis}]\label{theo:KinF}
Let $(P,\om)$ be a labeled poset with $|P|=n$.  Then
\[
\kpw = \sum_{\pi \in \L(P,\om)} F_{\des(\pi), n} = \sum_{\pi \in \L(P,\om)} F_{\co(\pi)}\,.
\]
\end{theorem}

\begin{example}\label{exa:kexpansion}
The labeled poset $(P,\om)$ of Figure~\ref{fig:labeled_poset} has $\L(P,\om) = \{1324, 1342\}$ and hence
\begin{align*}
\kpw ={}& F_{\{2\},4} + F_{\{3\},4} \\
={}&  F_{22} + F_{31} \\ 
= \begin{split}  {}&(M_{\{2\},4} + M_{\{1,2\},4} + M_{\{2,3\},4} + M_{\{1,2,3\},4}) \\{}& + (M_{\{3\},4} + M_{\{1,3\},4} + M_{\{2,3\},4} + M_{\{1,2,3\},4}) \end{split} \\
={}&  M_{22} + M_{31} + M_{112}  + 2M_{211} + M_{121} + 2M_{1111}.
\end{align*}
\end{example}

\section{Consequences of the poset viewpoint}\label{sec:poset_results}

As mentioned in the Introduction, an advantage of the poset viewpoint is that equality among $K_P$ has already been studied in the literature.  In this largely expository section, we gather together these results, especially those relevant to Conjecture~\ref{con:poset_trees}.  As the $K_P$ notation indicates, these results are confined to posets with all weak edges.

Results from the literature mostly fall into three classes: irreducibility, classes of posets within which we know that the $P$-partition enumerator distinguishes the posets, and necessary conditions on $P$ and $Q$ for $K_P = K_Q$.  

\subsection{Irreducibility}

If $P$ disconnects into two posets $P_1$ and $P_2$, then it follows from the definition of $K_P$ that $K_P = K_{P_1}K_{P_2}$.  On the other hand, a key result from \cite{LiWe21} is that if $P$ is connected, then $K_P$ is irreducible in \qsym.  Moreover, \qsym\ is known to be a unique factorization domain \cite{Haz01,LaPy08,MaRe95}.  As a consequence, Liu and Weselcouch deduce the following result.

\begin{corollary}\cite[Corollary~4.20]{LiWe21}\label{cor:factors}
For a poset $P$, the irreducible factorization of $K_P$ is given by $K_P = \prod_i K_{P_i}$, where the $P_i$ are the connected components of $P$.
\end{corollary}

Therefore, when studying $K_P = K_Q$, it suffices to consider the case when both $P$ and $Q$ are connected (see \cite[Corollary~4.21]{LiWe21}).  Additionally, a proof of Conjecture~\ref{con:poset_trees} would also mean that $K_P$ distinguishes forests.

Returning briefly to the setting of the chromatic quasisymmetric function $X_{\diG}$ we get the following consequence.  

\begin{proposition}
For directed acyclic graphs $\diG$ and $\diH$ with $X_{\diG}(\x,t) = X_{\diH}(\x,t)$, if $\diG$ is connected then so is $\diH$.
\end{proposition}

\begin{proof}
If $\diH$ has connected components $\diH_1, \ldots, \diH_r$ with $r\geq 2$, then it follows from Definition~\ref{def:xg} that $X_{\diH}(\x,t)$ factors as 
\[
X_{\diH}(\x,t) = X_{\diH_1}(\x,t) \cdots X_{\diH_r}(\x,t).
\]
Thus
\begin{equation}\label{equ:connected_digraph}
X_{\diG}(\x,t) = X_{\diH_1}(\x,t) \cdots X_{\diH_r}(\x,t).
\end{equation}

Now consider the irreducibility of the coefficient of the highest power of $t$ on both sides of~\eqref{equ:connected_digraph}.  On the left-hand side, this is $K_P$ for the poset $P$ corresponding to $\diG$, and $K_P$ is irreducible since $\diG$ and hence $P$ is connected.  On the right-hand side, this coefficient is the product of the coefficients of the highest powers of $t$ in each $X_{\diH_i}(\x,t)$, which is a contradiction since we know these coefficients are not constants.  
\end{proof}

\subsection{Distinguishing within classes of posets}\label{sub:classes}

The $P$-partition enumerator $K_P$ is known to distinguish posets within each of the following classes.  
\begin{itemize}
\item Rooted trees \cite{HaTs17,Zho20+}, i.e., posets that are trees with a single minimal element.
\item More generally, posets that are both bowtie-free and \textsf{N}-free \cite{HaTs17}.  As one would expect, \textsf{N} is the poset consisting of elements $a_1,a_2,b_1,b_2$ whose cover relations are $a_1 < b_1 > a_2 < b_2$, and a poset is \textsf{N}-free if it does not contain \textsf{N} as an induced subposet.
\item Series-parallel posets \cite[Theorem~5.2]{LiWe21}.  These can be defined in two ways.  They are the posets that can be formed by repeated operations of disjoint union and ordinal sum.  Equivalently, they are the \textsf{N}-free posets. 
\item Posets of width two \cite{LiWe20}.
\item Posets whose Greene shape is a hook \cite{LiWe20}.  The \emph{Greene shape} of a poset $P$ is the partition $(c_1-c_0, c_2-c_1, \ldots)$ where $c_i$ is the maximum cardinality of a union of $k$ chains of $P$.  So a poset whose Greene shape is the hook $(j, 1^i)$ has a maximal chain with $j$ elements and $i$ additional elements which form an antichain. 
\item Posets with Greene shape $(k, 2, 1, 1, \ldots, 1)$ for some $k\geq2$ \cite{LiWe20}.
\end{itemize}

\subsection{Necessary conditions for equality}

If $K_P = K_Q$, then all the statements in the next list hold. In the same way that knowledge of $K_P$ is equivalent to knowledge of $\KK{P}$, both are equivalent to knowledge of $K_{P^*}$ where $P^*$ is the dual of $P$ (see, for example, \cite[\S 3]{McWa14}).  Thus all the statements below have dual versions.

\begin{itemize}
\item $P$ and $Q$ obviously have the same number of vertices; if they are not trees, they need not have the same number of edges, as shown by Figure~\ref{fig:equal_kp}(a).
\item By Theorem~\ref{theo:KinF}, $P$ and $Q$ have the same number of linear extensions.
\item The \emph{jump} of an element $p$ of $P$ is defined to be the length (number of edges) of the longest chain from $p$ down to a minimal element of $P$.  Then $P$ and $Q$ have the same number of elements of jump $i$ for all $i$  \cite{McWa14}.  This can sometimes be a quick way to show that $K_P \neq K_Q$\,.

\item $K_{P_i} = K_{Q_i}$, where $P_i$ denotes the induced subposet of $P$ consisting of elements of jump at least $i$ \cite{McWa14}.  For example, with $i=1$, we get that $K_{P^-} = K_{Q^-}$, where $P^-$ denotes the result of deleting all the minimal elements from $P$.

\item The \emph{up-jump} of $p$ denotes the length of the longest chain from $p$ to a maximal element, and define the \emph{jump-pair} of $P$ to be (jump of $p$, up-jump of $p$).  Then for all $i$ and $j$, the number of elements with jump pair $(i,j)$ is the same for $P$ as for $Q$ \cite{LiWe20}.

\item Let $\anti_{k,i,j}(P)$ denote the number of $k$-element order ideals $I$ of $P$ such that $I$ has $i$ maximal elements and $P\setminus I$ has $j$ minimal elements.  Then $\anti_{k,i,j}(P) = \anti_{k,i,j}(Q)$ for all $i,j,k$ \cite{LiWe20}.


\item Summing over $j$ and $k$ in the previous item, we get that $P$ and $Q$ have the same number of antichains of each size, as conjectured in \cite{McWa14} and shown in \cite{LiWe20}.

\item Suppose that for some $k$ and $i$, $P$ has a unique order ideal $I_P$ of size $k$ with $i$ maximal elements, and similarly for an order ideal $I_Q$ of $Q$.  Then $K_{P\setminus I_P}$ can be determined from $K_P$ \cite[Corollary 3.6]{LiWe20} and hence $K_{P\setminus I_P} = K_{Q \setminus I_Q}$.

\item Summing over $j$ in $\anti_{k,i,j}(P)$ shows that the number of order ideals of size $k$ with $i$ maximal elements has to be the same for $P$ as for $Q$.  Then there are various ways we can combine the results above.  For example, the number of elements with principal order ideal of size $k$ and up-jump $j$ is the same for $P$ as for $Q$ \cite{LiWe20}.

\item The Greene shape of $P$ equals that of $Q$ \cite{LiWe20}.

\item A $P$-partition $f$ is \emph{pointed} if $f$ is surjective onto $[k]$ for some $k$ and $f^{-1}(i)$ has a unique minimal element for all $i$ with $1 \leq i \leq k$.  The weight of $f$ is the composition $\wt(f) = (\#f^{-1}(1), \#f^{-1}(2), \ldots)$.  The number of pointed $P$-partitions of any given weight is the same for $P$ and $Q$.  This follows immediately from the result of \cite{AlSu21} that if we expand $K_P$ in the (unnormalized) power sum basis $\psi_\alpha$ of type I, then the coefficient of $\psi_\alpha$ equals the number of pointed $P$-partitions of weight $\alpha$ .



\item \cite[Lemma~4.10]{LiWe20} shows that any finite poset $P$ has a unique antichain $A$ of maximum size such that any other antichain of maximum size is contained in the order ideal $I(A)$ generated by $A$.  Let $P^-$ be the subposet consisting of elements less than $A$ in $P$.  Then since $K_{P^-}$ is determined by $K_P$ (also shown in \cite{LiWe20}), we must have $K_{P^-} = K_{Q^-}$.  Similarly for the subposet $P^+$ above $A$.
\end{itemize}

If one is given non-isomorphic trees $T_1$ and $T_2$, it is typically straightforward to find a result on the list above that will show that they have unequal $P$-partition enumerators.  However, the problem is that the result used will depend on $T_1$ and $T_2$, i.e., we don't have a systematic way.

Trees of rank one are difficult enough that they are a good test case for techniques; see Figure~\ref{fig:rank_one} for a non-isomorphic pair.  We can use $\anti_{k,1,j}(P)$ and $ \anti_{k,1,j}(P^*)$ to determine the degree sequences for the maximal and minimal elements, respectively; these match up in the figure.  To distinguish the pair in the figure, we can use pointed $P$-partitions: the tree on the right has a pointed $P$-partition of weight $(4,1,4,2)$ but the tree on the left does not.

\begin{figure}
\begin{center}
\begin{tikzpicture}[scale=0.7] 
\tikzstyle{every node}=[draw, shape=circle, inner sep=2pt]; 
\begin{scope}  
\draw (-2,0) node (a1) {};
\draw (-0.67,0) node (a2) {};
\draw (0.67,0) node (a3) {};
\draw (2,0) node (a4) {};
\draw (-3,2) node (b1) {};
\draw (-2,2) node (b2) {};
\draw (-1,2) node (b3) {};
\draw (0,2) node (b4) {};
\draw (1,2) node (b5) {};
\draw (2,2) node (b6) {};
\draw (3,2) node (b7) {};
\draw (a1) -- (b1);
\draw (a1) -- (b2);
\draw (a1) -- (b3) -- (a2);
\draw (b3) -- (a3) -- (b4);
\draw (a4) -- (b4);
\draw (a4) -- (b5);
\draw (a4) -- (b6);
\draw (a4) -- (b7);
\end{scope}  

\begin{scope}[xshift=22em] 
\draw (-2,0) node (a1) {};
\draw (-0.67,0) node (a2) {};
\draw (0.67,0) node (a3) {};
\draw (2,0) node (a4) {};
\draw (-3,2) node (b1) {};
\draw (-2,2) node (b2) {};
\draw (-1,2) node (b3) {};
\draw (0,2) node (b4) {};
\draw (1,2) node (b5) {};
\draw (2,2) node (b6) {};
\draw (3,2) node (b7) {};
\draw (a1) -- (b1);
\draw (a1) -- (b2);
\draw (a1) -- (b3) -- (a2);
\draw (a2) -- (b4) -- (a3);
\draw (a4) -- (b4);
\draw (a4) -- (b5);
\draw (a4) -- (b6);
\draw (a4) -- (b7);
\end{scope}
\end{tikzpicture}
\end{center}
\caption{Two trees of rank one with the same degree sequences but different $K_P$}
\label{fig:rank_one}
\end{figure}
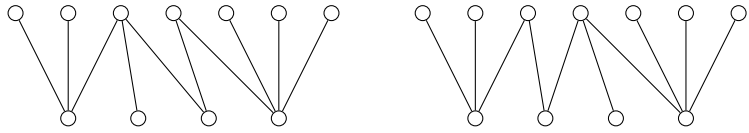

\section{Adding strict edges}\label{sec:strict_weak_edges}

This section considers extending Conjecture~\ref{con:poset_trees} in the following way, as inspired by \cite[Problem~6.2]{HaTs17}: does $\kpw$ distinguish labeled trees when we allow any mixture of strict and weak edges?  The answer is ``no'' in general as shown, for example, by the labeled trees in Figure~\ref{fig:trees_equal_kpw}.  

In fact, this question connects to a studied one in the realm of symmetric functions.  Semistandard Young tableaux of a skew shape $\lambda/\mu$ can be considered as $(P,\om)$-partitions of a particular labeled poset; see \cite[\S7.19]{Sta99}, for example.  In this case, $\kpw$ equals the skew Schur function $s_{\lambda/\mu}$\,.  When $\lambda/\mu$ is a \emph{ribbon}, meaning it is connected and has no 2-by-2 block of cells, the corresponding $(P,\om)$ will be a tree.  It is well known that $s_{\lambda/\mu}$ is invariant under rotation of $\lambda/\mu$ by 180$^\circ$ thus yielding an infinite class of pairs of trees with the same $\kpw$; the simplest example is in Figure~\ref{fig:trees_equal_kpw}(b).  It is natural to ask if they are other pairs of ribbons, unequal under 180$^\circ$ rotation, that give rise to the same skew Schur function. The answer is ``yes'' and the full classification of such pairs is given in \cite{BTvW06}.

On the other hand, Hasebe and Tsujie \cite{HaTs17} have shown that $\kpw$ distinguishes \emph{rooted} trees with all strict (equivalently all weak) edges.  In Conjecture~\ref{con:labeled_rooted_trees}, we propose that their result also holds for arbitrary labeled posets that are rooted trees.  In other words, if we want $\kpw$ to distinguish trees with a mixture of strict and weak edges, restricting to rooted trees works.  We have verified Conjecture~\ref{con:labeled_rooted_trees} for $n\leq 10$.

Although we have not succeeded in resolving Conjecture~\ref{con:labeled_rooted_trees}, the rest of this section focuses on results that still significantly generalize those in \cite{HaTs17}, as we next explain.   

\begin{figure}
\begin{center}
\begin{tikzpicture}[scale=0.7] 
\begin{scope}
\tikzstyle{every node}=[draw, shape=circle, inner sep=2pt]; 
\begin{scope}  
\draw (1,0) node (a1) {};
\draw (3,0) node (a2) {};
\draw (0,1) node (a3) {};
\draw (2,1) node (a4) {};
\draw (4,1) node (a5) {};
\draw (2,2.4) node (a6) {};
\draw (a5) -- (a2) -- (a4) -- (a6);
\draw (a1) -- (a3);
\draw[double distance=2pt] (a1) -- (a4);
\end{scope}  
\begin{scope}[yshift=-10em] 
\draw (2,0) node (a1) {};
\draw (0.4,1) node (a2) {};
\draw (2,1.4) node (a3) {};
\draw (3.6,1) node (a4) {};
\draw (0.4,2.4) node (a5) {};
\draw (3.6,2.4) node (a6) {};
\draw (a1) -- (a3);
\draw (a1) -- (a2) -- (a5);
\draw (a4) -- (a6);
\draw[double distance=2pt] (a1) -- (a4);
\end{scope}
\end{scope}
\begin{scope}[yshift=-13em] 
\draw (2,0) node {(a)};
\end{scope}

\begin{scope}[xshift=20em,yshift=-4em]
\begin{scope}  
\tikzstyle{every node}=[draw, shape=circle, inner sep=2pt]; 
\draw (1,0) node (a1) {};
\draw (0,1) node (a2) {};
\draw (2,1) node (a3) {};
\draw (a1) -- (a3);
\draw[double distance=2pt] (a1) -- (a2);

\begin{scope}[yshift=-6em] 
\draw (0,0) node (a1) {};
\draw (2,0) node (a2) {};
\draw (1,1) node (a3) {};
\draw (a1) -- (a3);
\draw[double distance=2pt] (a2) -- (a3);
\end{scope}
\end{scope}
\begin{scope}[yshift=-9em] 
\draw (1,0) node {(b)};
\end{scope}
\end{scope}
\end{tikzpicture}
\end{center}
\caption{Pairs of labeled trees with the same $\kpw$}
\label{fig:trees_equal_kpw}
\end{figure}
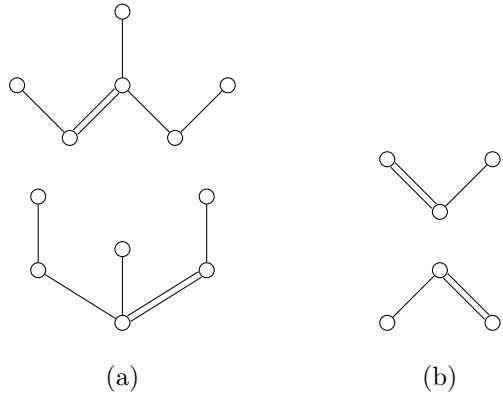

\subsection{Fair trees}

We consider rooted trees as being rooted at the bottom, and thus a child is above its parent. 

\begin{definition}
A {\em fair tree} is a labeled poset $(P,\om)$ such that:
\begin{itemize}
\item the underlying poset $P$ is a rooted tree,
\item for each vertex $v$ in $(P,\om)$, its outgoing edges (to its children) are either all strict or all weak.
\end{itemize} 
\end{definition}
Figure \ref{fig:fair-tree-ex} shows an example of a fair tree.  The ``fair'' terminology comes from the idea that each parent is equally strict with all its children.  

\begin{figure}
\begin{center}
\begin{tikzpicture}[scale=0.7] 
\tikzstyle{every node}=[draw, shape=circle, inner sep=2pt]; 

\draw (0,0) node (a1) {};

\draw (-2.5,1) node (a2) {};
\draw (0,1) node (a3) {};
\draw (2.5,1) node (a4) {};

\draw (-3,2) node (a5) {};
\draw (-2,2) node (a6) {};
\draw (-1,2) node (a7) {};
\draw (0,2) node (a8) {};
\draw (1,2) node (a9) {};
\draw (2.5,2) node (a10) {};

\draw (-2,3) node (a11) {};
\draw (2,3) node (a12) {};
\draw (3,3) node (a13) {};

\draw (a1) -- (a2);
\draw (a1) -- (a3);
\draw (a1) -- (a4);
\draw (a2) -- (a5);
\draw (a2) -- (a6);
\draw (a10) -- (a12);
\draw (a10) -- (a13);

\draw[double distance=2pt] (a6) --(a11);
\draw[double distance=2pt] (a3) --(a7);
\draw[double distance=2pt] (a3) --(a8);
\draw[double distance=2pt] (a3) --(a9);
\draw[double distance=2pt] (a4) --(a10);

\end{tikzpicture}
\end{center}
\caption{Example of a fair tree of size $13$}
\label{fig:fair-tree-ex}
\end{figure}
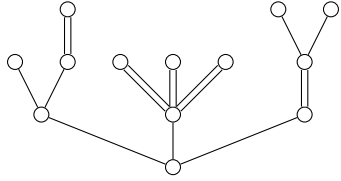

We will prove that the fair trees are distinguished by the $(P,\om)$-partition enumerator $\kpw$.  In fact, we shall consider a wider class $\Cb$ which we introduce next.  

Let us define the following (noncommutative) operations on labeled posets.
For $(P,\om)$ and $(Q,\om')$ two labeled posets (considered as posets with assignments of strict and weak edges),
the \emph{weak (resp.\ strict) ordinal sum} $(P,\om)\wP (Q,\om')$ (resp.\ $(P,\om)\sP (Q,\om')$) is the labeled poset obtained 
by placing $(P,\om)$ below $(Q,\om')$ 
and adding a weak (resp.\ strict) edge from each maximal element of $P$ to each minimal element of $Q$.  

It will be helpful in what follows to pick an explicit labeling of the elements of these ordinal sums that is consistent with the strictness of the edges.  We will label $(P,\om)\wP (Q,\om')$ by copying over the labels from $(P,\om)$ and $(Q,\om')$ but increasing each of the $\om'$-labels by $|P|$ so that the labeling is a bijection to $[|P|+|Q|]$.  We can label the disjoint union $(P,\om)\sqcup (Q,\om')$ in the same way.  Similarly, $(P,\om)\sP (Q,\om')$ will be labeled by instead increasing each of the $\om$-labels by $|Q|$.

\begin{definition}\label{def:C}
We define the set $\Cb$ of labeled posets recursively by:
\begin{enumerate}
\item the one-element labeled poset $[1]$ is in $\Cb$;
\item for any $(P,\om)$ and $(Q,\om')$ in $\Cb$, their disjoint union $(P,\om)\sqcup (Q,\om')$ is in $\Cb$;
\item for any $(P,\om)$ in $\Cb$, the ordinal sums
$[1]\wP (P,\om)$ and $[1]\sP (P,\om)$
are in $\Cb$;
\item for any $(P,\om)$ in $\Cb$, the ordinal sums
$(P,\om)\wP [1]$ and $(P,\om)\sP [1]$ 
are in $\Cb$.
\end{enumerate} 
\end{definition}
Figure \ref{fig:Cb-ex} shows an example.  See Subsection~\ref{sub:Cv2} for a characterization of $\Cb$ in terms of forbidden subposets.

\begin{figure}
\begin{center}
\begin{tikzpicture}[scale=0.7] 
\tikzstyle{every node}=[draw, shape=circle, inner sep=2pt]; 

\draw (0,-1) node (a1) {};

\draw (-5,-2) node (a2) {};
\draw (-1.5,-2) node (a3) {};

\draw (-6,-3) node (a4) {};
\draw (-5,-3) node (a5) {};
\draw (-4,-3) node (a6) {};
\draw (-2,-3) node (a7) {};
\draw (-1,-3) node (a8) {};
\draw (1,-3) node (a9) {};
\draw (3,-3) node (a10) {};
\draw (4,-3) node (a11) {};

\draw (-5.5,-4) node (a12) {};
\draw (3.5,-4) node (a13) {};

\draw (-3,-5) node (a14) {};

\draw  (5.5, -3.5) node (a15) {};
\draw  (5.5, -2.5) node (a16) {};

\draw (a1) -- (a2);
\draw (a1) -- (a3);
\draw (a1) -- (a9);
\draw (a1) -- (a10);
\draw (a1) -- (a11);
\draw (a3) -- (a7);
\draw (a3) -- (a8);
\draw (a4) -- (a12);
\draw (a5) -- (a12);

\draw[double distance=2pt] (a2) --(a4);
\draw[double distance=2pt] (a2) --(a5);
\draw[double distance=2pt] (a2) --(a6);
\draw[double distance=2pt] (a10) --(a13);
\draw[double distance=2pt] (a11) --(a13);
\draw[double distance=2pt] (a12) --(a14);
\draw[double distance=2pt] (a6) --(a14);
\draw[double distance=2pt] (a7) --(a14);
\draw[double distance=2pt] (a8) --(a14);
\draw[double distance=2pt] (a15) --(a16);

\end{tikzpicture}
\end{center}
\caption{Example of an element of $\Cb$ of size $16$}
\label{fig:Cb-ex}
\end{figure}
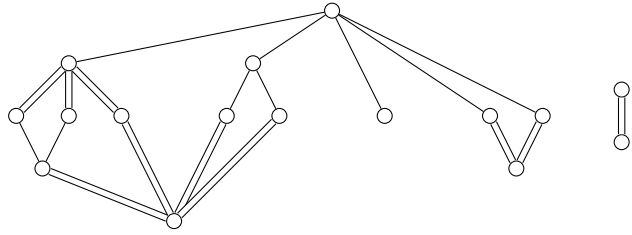

\begin{remark}
We can use Definition~\ref{def:C} to give an alternative and recursive definition of fair trees.
Define {\em fair forests} as the class defined recursively by (a)--(c) in Definition \ref{def:C}.  Fair trees are nothing but connected fair forests, thus are elements of $\Cb$.
\end{remark}

The main result in this section is the following.
It is a generalization of Theorems~1.3 and~5.1 in \cite{HaTs17}; see Proposition~\ref{prop:Cb-pat} below for the analogue of Hasebe and Tsujie's bowtie- and \textsf{N}-free characterization.

\begin{theorem}\label{thm:fair-trees}
The $(P,\om)$-partition enumerator $\kpw$ distinguishes elements in $\Cb$, 
thus in particular fair trees. More formally, for labeled posets $(P,\om)$ and $(Q,\om')$ in $\Cb$,
the following assertions are equivalent:
\begin{enumerate}
\item $(P,\om)$ and $(Q,\om')$ are isomorphic;
\item $\kpw=K_{(Q,\om')}$.
\end{enumerate}
\end{theorem}

The crux of the proof is the irreducibility result Proposition~\ref{prop:irred-C}, in the same way that irreducibility played a key role in \cite{HaTs17,LiWe21}.  We first need more background on \qsym.

\subsection{Products of quasisymmetric functions}\label{sub:prods_qsym}

It will help our intuition to recall how to interpret $F_\alpha$ as a $(P,\om)$-partition enumerator.  If $\alpha$ is a composition of $n$, we let $P$ be the chain with $n$ elements, labeled from top to bottom by any permutation $\pi$ of $[n]$ such that $\co(\pi) = \alpha$.  That the resulting $\kpw = F_\alpha$ follows directly from Theorem~\ref{theo:KinF}; see Figure~\ref{fig:f_poset} for an example.  Thinking just in terms of strict and weak edges, for a general $\alpha$, simply insert the strict edges so that the numbers of elements in the chains of contiguous weak edges, from bottom to top, match the parts of $\alpha$.
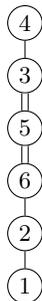
\begin{figure}
\begin{center}
\begin{tikzpicture}[scale=0.7] 
\tikzstyle{every node}=[draw, shape=circle, inner sep=2pt]; 
\draw (0,0) node (a1) {\small 1};
\draw (0,1) node (a2) {\small 2};
\draw (0,2) node (a3) {\small 6};
\draw (0,3) node (a4) {\small 5};
\draw (0,4) node (a5) {\small 3};
\draw (0,5) node (a6) {\small 4};
\draw (a1) -- (a2) -- (a3);
\draw (a5) -- (a6);
\draw[double distance=2pt] (a3) --(a4) -- (a5);
\end{tikzpicture}
\end{center}
\caption{This labeled poset has $(P,\om)$-partition enumerator $F_{312}\,.$}
\label{fig:f_poset}
\end{figure}

From~\eqref{equ:kp}, we know that labeled posets $(P,\om)$ and $(Q,\om')$  satisfy 
\[
\kpw K_{(Q,\om')} = K_{(P,\om) \sqcup (Q,\om')}.
\]
Thus we can interpret the product $F_\alpha F_\beta$ as a $(P,\om)$-partition enumerator for a disjoint union of two appropriately labeled chains.  One fact that we need later from this interpretation is summarized in the following lemma.

\begin{lemma}\label{lem:deg2}
Let $f_1$ and $f_2$ be two non-constant elements of \qsym.
Then the $F$-support of $f_1 f_2$ contains at least one composition $\gamma$ whose first part $\gamma_1$ is at least $2$.
The same is true when focusing on the last part of the compositions appearing in the $F$-support of $f_1 f_2$.
\end{lemma}
\begin{proof}
First consider two non-empty compositions $\alpha$ and $\beta$.  Construct two labeled chains $(P,\om)$ and $(Q,\om')$ such that $F_\alpha=K_{(P,\om)}$ and $F_\beta=K_{(Q,\om')}$.  Thus $F_\alpha F_\beta = K_{(P,\om) \sqcup (Q,\om')}$.  It is obvious that there is at least one element $\pi$ in the set $\L((P,\om) \sqcup (Q,\om'))$ with $\pi_1<\pi_2$.  By Theorem~\ref{theo:KinF}, this gives rise to an element $F_\gamma$ in the $F$-support of $F_\alpha F_\beta$ with $\gamma_1\ge2$.

Now, let us consider two non-constant elements $f_1$ and $f_2$ of \qsym.
Let us denote by $\alpha$ (resp.\ $\beta$) the lexicographically maximal composition in the $F$-support of $f_1$ (resp.\ $f_2$).
It is clear that the lexicographically maximal element in the product $f_1 f_2$ comes from the product $F_\alpha F_\beta$,
thus the case above applies.

The case of the last part is similar, requiring just two tweaks.  First, if $\ell$ denotes the number of elements in $(P,\om) \sqcup (Q,\om')$, it is obvious that there is at least one element $\pi$ in the set $\L((P,\om) \sqcup (Q,\om'))$ with $\pi_{\ell-1}<\pi_\ell$.   This yields an element $F_\gamma$ in the $F$-support of $F_\alpha F_\beta$ such that the last part of $\gamma$ is at least 2. Secondly, instead of considering the lexicographically maximal composition, we use the composition whose reversal is lexicographically maximal.
\end{proof}

Let us introduce two operations on compositions 
(which are already known but we shall use notation relevant to our context):
\begin{equation}\label{equ:uparrow}
(\alpha_1,\alpha_2,\dots,\alpha_k) \wP (\beta_1,\beta_2,\dots,\beta_\ell)
=
(\alpha_1,\alpha_2,\dots,\alpha_k+\beta_1,\beta_2,\dots,\beta_\ell),
\end{equation}
and
\begin{equation}\label{equ:upuparrow}
(\alpha_1,\alpha_2,\dots,\alpha_k) \sP (\beta_1,\beta_2,\dots,\beta_\ell)
=
(\alpha_1,\alpha_2,\dots,\alpha_k,\beta_1,\beta_2,\dots,\beta_\ell),
\end{equation}
which give rise to two (noncommutative) products in \qsym, defined on the $F$-basis:
\begin{equation}\label{equ:f_uparrow}
F_{\alpha} \wP F_{\beta}
=
F_{\alpha \wP \beta},
\end{equation}
and
\begin{equation}\label{equ:f_upuparrow}
F_{\alpha} \sP F_{\beta}
=
F_{\alpha \sP \beta}\,.
\end{equation}

The use of the same notation as for labeled posets is justified by the following statement.
\begin{proposition}\label{prop:k_products}
For any two labeled posets $(P,\om)$ and $(Q,\om')$, 
$$K_{(P,\om)\wP(Q,\om')}=\kpw \wP K_{(Q,\om')}$$
and
$$K_{(P,\om)\sP(Q,\om')}=\kpw \sP K_{(Q,\om')}.$$
\end{proposition}
\begin{proof}
Let us prove the first assertion.
By Theorem~\ref{theo:KinF}, 
\[
K_{(P,\om) \wP (Q,\om')} = \sum_{\pi\in\L((P,\om)\wP (Q,\om'))} F_{\co(\pi)}.
\]
By definition of $\wP$ for labeled posets, those $\pi\in\L((P,\om)\wP (Q,\om'))$ are concatenations of $\sigma\in\L(P,\om)$ and $\tau\in\L(Q,\om')$ but where each entry of $\tau$ is increased by $|P|$.   Consequently, $\co(\pi) = \co(\sigma)\wP \co(\tau)$. Thus, using \eqref{equ:f_uparrow}, we have 
\begin{align*}
K_{(P,\om) \wP (Q,\om')} 
&= \sum_{\sigma\in\L(P,\om),\:\tau\in\L(Q,\om')} F_{\co(\sigma)\wP \co(\tau)}\\
&= \sum_{\sigma\in\L(P,\om),\:\tau\in\L(Q,\om')} F_{\co(\sigma)}\wP F_{\co(\tau)}\\
&= \kpw \wP K_{(Q,\om')}\ .
\end{align*}
The second assertion is proved similarly.
\end{proof}

\subsection{An involution on labeled posets}

Given a labeled poset $(P,\om)$, we can switch the strictness of the edges to obtain a new labeled poset $\inv{(P,\om)}$.  
We will follow \cite{McWa14} by referring to this operation on labeled posets as the \emph{bar operation}.  In the setting of quasisymmetric functions, we may define 
\[
\inv{F_\alpha} = \inv{F_{S(\alpha),n}} = F_{\inv{S(\alpha)},n}
\]
where for any subset $S$ of $[n-1]$ we let $\inv{S}=[n-1]\backslash S$,  and we extend it to \qsym\ by linearity.  In the example of Figure~\ref{fig:f_poset}, we get $\inv{F_{312}} = \inv{F_{\{3,4\},6}} = F_{\{1,2,5\},6}= F_{1131}$.
The following lemma states that these two bar operations are compatible, and that the latter operation commutes with the product in \qsym.
\begin{lemma}\label{lem:bar}
We have:
\begin{enumerate}
\item $\inv{\kpw} = K_{\inv{(P,\om)}}$ for any labeled poset $(P,\om)$;
\item $\inv{fg}=\inv{f}\inv{g}$ for any $f$ and $g$ in \qsym.
\end{enumerate}
\end{lemma}
\begin{proof}
For the first assertion, observe that the bar operation that sends $(P,\om)$ to $\inv{(P,\om)}$ on general labeled posets can be done at the level of the labeling $\om$ by simply replacing each label $\omega(i)$ by $|P|+1-\om(i)$.  Then in the notation of the $F$-expansion of $\kpw$ in Theorem~\ref{theo:KinF}, each $\des(\pi)$ is sent to its complement $[n-1]\backslash\des(\pi)$, resulting in $\inv{\kpw}$, as required.

For the second assertion, recall the interpretation of $F_\alpha$ as $\kpw$ for a labeled chain $(P,\omega)$ from the start of Subsection~\ref{sub:prods_qsym}, along with the interpretation there of $F_\alpha F_\beta$.
Thus (a) implies that $\inv{F_\alpha F_\beta} = \inv{F_\alpha} \; \inv{F_\beta}$, from which (b) follows by linearity.
\end{proof}

\subsection{Irreducibility}

A crucial fact towards proving Theorem \ref{thm:fair-trees} is the irreducibility 
of $\kpw$ for elements of $\Cb$. 

\begin{proposition}\label{prop:irred-C}
If $(P,\om)$ is a connected element of $\Cb$ then $\kpw$ is irreducible in \qsym.
\end{proposition}

To prove this, we first recall the following general property.  A polynomial with integer coefficients is said to be \emph{primitive} if 
whenever an integer $k$ divides all its coefficients we have $k=\pm1$.
In the same vein, since we consider $\qsym \subseteq \mathbb{Z}[[x_1, x_2, \ldots]]$, we say $f \in \qsym$ is \emph{primitive} if whenever an integer $k$ divides $f$ in \qsym\ we have $k = \pm1$.
We are interested in whether $\kpw$ is primitive.  

To answer this question, let us define the \emph{leading term} of a formal power series $f$ expanded in terms of monomials as the term $c_\alpha x^\alpha = c_\alpha x_1^{\alpha_1} x_2^{\alpha_2} \cdots$ with lexicographically largest $\alpha$ such that $c_\alpha \neq 0$; naturally, we call this $c_\alpha$ the \emph{leading coefficient}.  For a labeled poset $(P,\om)$, we define $x^{\mathrm{jump}(P,\om)} = x_1^{j_1} x_2^{j_2} \cdots$ where $j_i$ is the number of elements of $(P,\om)$ of jump $i-1$.

\begin{proposition}[{\cite[proof of Proposition 4.2]{McWa14}}]\label{prop:prim}
For any labeled poset $(P,\om)$, the leading term of $\kpw$ is $x^{\mathrm{jump}(P,\om)}$.  In particular, the leading coefficient of $K_{(P, \om)}$ is $1$ and $K_{(P, \om)}$ is primitive.
\end{proposition}

The next lemma is relevant to the recursive construction of $\Cb$.

\begin{lemma}\label{lem:irred}
If $f$ is a primitive element of \qsym\  then the polynomials
$F_1\wP f$,\ $F_1\sP f$,\ $f\wP F_1$, and $f\sP F_1$
are irreducible in \qsym.
\end{lemma}
\begin{proof}
Let us first deal with $F_1\sP f$.
Assume that $F_1\sP f$ is reducible. 
Note that if $f$ expands as $\sum c_\alpha F_\alpha$ then $F_1\sP f = \sum c_\alpha F_{\alpha^+}$, where $\alpha^+$ is obtained from $\alpha$ by appending an entry of 1 at the start.  Thus since $f$ is primitive, $F_1\sP f$ is also primitive, so there exist non-constants $g,g'\in \qsym$, such that 
\begin{equation}\label{eq:lem-irred1}
F_1\sP f = g g'.
\end{equation}
All $\alpha$ in the $F$-support of $F_1\sP f$ satisfy $\alpha_1=1$ but by Lemma~\ref{lem:deg2}, the $F$-support of the right-hand side of~\eqref{eq:lem-irred1} contains at least one $\alpha$ with $\alpha_1 \geq 2$, a contradiction.  The case of  $f\sP F_1$ is treated similarly, using the last part of $\alpha$ instead of the first part.

Let us now consider $F_1\wP f$.
We again assume $F_1\wP f$ is reducible and now $\alpha^+$ is obtained from $\alpha$ by increasing the first part by 1.  The primitivity 
implies the existence of two non-constants $g,g'\in \qsym$, 
such that $F_1\wP f = g g'$. We apply the bar operator to this equality
and get 
\[
\inv{F_1\wP f} = \inv{g g'} = \inv{g} \inv{g'}.
\]
by Lemma~\ref{lem:bar}(b). Since the elements $\alpha$ of the $F$-support of $F_1\wP f$ all satisfy $\alpha_1\ge 2$, we know all $\alpha'$ in the $F$-support of 
$\inv{F_1\wP f}$ satsify $\alpha'_1 = 1$.
We are thus led to the same contradiction as in the first case. 
The case of  $f\wP F_1$ is treated similarly.
\end{proof}

Since we will only be considering labeled posets in $\Cb$ for the remainder of this section, we will abbreviate $(P,\om)$ as $P$ for easier reading.

\begin{proof}[Proof of Proposition \ref{prop:irred-C}]
If $P$ has just a single element, then $K_P = F_1$, which is irreducible.  Otherwise, since $P$ is connected and constructed recursively, it must be of one of the forms $[1]\wP P'$,\ $[1]\sP P'$,\ 
$P'\wP [1]$, or $P'\sP [1]$, where $P' \in \Cb$.  The result now follows from Propositions~\ref{prop:k_products} and~\ref{prop:prim}, and Lemma~\ref{lem:irred}.
\end{proof}

\begin{remark}
A special case of Proposition~\ref{prop:irred-C}, or even just Lemma~\ref{lem:irred}, is that $F_\alpha$ is irreducible in \qsym, as previously shown in~\cite{LaPy08}.
\end{remark}

We are now in a position to prove the main result of this section.  

\begin{proof}[Proof of Theorem \ref{thm:fair-trees}]
It is obvious that when $P$ and $Q$ are isomorphic, then $K_P=K_Q$.
Let us prove the converse.

We shall use induction on the size of $P$, with the case of size 1 being trivial.

Assuming now that the size of $P$ is at least $2$, 
we may decompose $P$ and $Q$ uniquely into non-empty
connected components:
$P=\sqcup_{i=1}^r P_i$
and
$Q=\sqcup_{i=1}^s Q_i$.
We have
\[
\prod_{i=1}^r K_{P_i}
=
\prod_{i=1}^s K_{Q_i}.
\]
Since \qsym\  is a unique factorization domain \cite{Haz01,LaPy08,MaRe95}, 
the irreducibility of $K_{P_i}$ and $K_{Q_i}$ from Proposition~\ref{prop:irred-C} 
implies that $r=s$ and that for every $i$, we have $K_{P_i} = K_{Q_i}$ (up to a suitable renumbering).

When $r \ge 2$, the size of each $P_i$ and $Q_i$ is smaller than the size of $P$
and thus $P_i$ and $Q_i$ are isomorphic for every $i$ by the induction hypothesis. 
Thus $P$ and $Q$ are also isomorphic. 

Suppose now that $r = 1$, i.e., $P$ and $Q$ are connected. Thus, 
since their size is greater than $1$, they may be written as
$P=[1]\wP P'$ or $P=[1]\sP P'$
or
$P=P'\wP [1]$ or $P=P' \sP [1]$,
and similarly for $Q$.
By Theorem~\ref{theo:KinF}, we can distinguish among these four possibilities by observing whether the first part of the compositions $\alpha$ in the $F$-support of $K_P$ (equivalently $K_Q$) are all equal to $1$ or all greater than $1$, and similarly for the last part of all such $\alpha$. Specifically, $P=[1]\wP P'$ if and only if $\alpha_1 > 1$ for all $\alpha$, implying $Q=[1]\wP Q'$.  Similarly, $P=[1]\sP P'$ if and only if $\alpha_1=1$ for all $\alpha$, implying $Q=[1]\sP Q'$.  And $P=P'\wP [1]$ (resp.\ $P=P' \sP [1]$) if and only if the last part of $\alpha$ is greater than 1 (resp.\ equals 1) for all $\alpha$, implying $Q=Q'\wP [1]$ (resp.\ $Q=Q' \sP [1]$).  
So for a given $P$ and $Q$ with $K_P = K_Q$\,, at least one of the following four properties holds:
\begin{itemize}
\item $P=[1]\wP P'$ and $Q=[1]\wP Q'$, or 
\item $P=[1]\sP P'$ and $Q=[1]\sP Q'$, or
\item $P=P'\wP [1]$ and $Q=Q'\wP [1]$, or 
\item $P=P'\sP [1]$ and $Q=Q'\sP [1]$.
\end{itemize}

In the first case, applying Proposition~\ref{prop:k_products} gives $F_1\wP K_{P'} = F_1\wP K_{Q'}$.  Let us expand this in the $F$-basis as $\sum c_\alpha F_\alpha$.  From~\eqref{equ:f_uparrow} and~\eqref{equ:uparrow}, we see that by just subtracting $1$ from the first part $\alpha_1$ in each term $F_\alpha$, we get $K_{P'} = K_{Q'}$, with $P'$ of size smaller than $P$.
By induction, $P'$ and $Q'$ are isomorphic, and hence so are $P$ and $Q$.

In the second case, we have  $F_1\sP K_{P'} = F_1\sP K_{Q'}$.
Let us expand this in the $F$-basis as $\sum c_\alpha F_\alpha$.
From~\eqref{equ:f_upuparrow} and~\eqref{equ:upuparrow}, we see that by just  removing the first part $\alpha_1=1$ 
in each term $F_\alpha$, we get $K_{P'} = K_{Q'}$, with $P'$ of size smaller than $P$, giving the same conclusion.

The last two cases are resolved similarly by focusing on the last entry of $\alpha$.
\end{proof}

\subsection{A different characterization of $\Cb$.}\label{sub:Cv2}

To end this section, we shall give a characterization of the class $\Cb$
in terms of forbidden subposets.

We need to distinguish two notions of subposets of labeled posets.
For the first, we ignore the labeling and consider the usual notion of induced subposets.
A poset (labeled or not) that avoids a set of unlabeled subposets $S$ in this sense is said to be $(S)$-free.
The second notion is of convex labeled subposets, considered as convex subposets with specific assignments of strict and weak edges.
A labeled poset that avoids a set of labeled convex subposets $S'$ in this sense is said to be $[S']$-free.
Of course these two notions can be used together.

The following result is the analogue in our context of \cite[Theorem 4.3]{HaTs17}.
\begin{proposition}\label{prop:Cb-pat}
Labeled posets in $\Cb$ are exactly $\pac$-$\pef$-free posets.
\end{proposition}

The proof is modeled on the proof of \cite[Theorem 4.3]{HaTs17}.  We first recall \cite[Lemma 4.4]{HaTs17}.
\begin{lemma}\label{lem:unique-max-min}
A finite connected $\pac$-free poset has a unique minimal or maximal element.
\end{lemma}

\begin{proof}[Proof of Proposition \ref{prop:Cb-pat}]
Let $P$ be an element of $\Cb$. 
We shall prove that it is $\pac$-$\pef$-free by induction on its size.
If $P$ is disconnected then $P=P' \sqcup P''$ with $P'$ and $P''$ in $\Cb$ and non-empty.
By the induction hypothesis, $P'$ and $P''$ are $\pac$-$\pef$-free, and thus so is $P$. 
If $P$ is connected, then by definition of $\Cb$, $P$ is of the form 
$[1]\wP P'$ or $[1]\sP P'$ or $P'\wP [1]$ or $P' \sP [1]$ for some $P'$ in $\Cb$.
We may assume without loss of generality that $P=[1] \wP P'$.
The induction hypothesis shows that $P'$ is $\pac$-$\pef$-free. 
This implies that $P$ is $\pac$-free too; otherwise, we would find three elements $a, b, c$ in $P'$ such that the subposet $(\{1, a, b, c\}, \le_P )$ is isomorphic to $\pa$ or $\pc$, 
but this is absurd since $\pa$ and $\pc$ do not have a unique minimal element.
A similar idea implies that $P$ is $\pef$-free too; otherwise, we would find two elements $a, b$ in $P'$ such that the convex labeled subposet $(\{1, a, b\}, \le_P )$ is equal to $\pe$ or $\pf$, but this is impossible by construction.

Let us then prove the converse.
The proof is again based on induction on the size of $P$. 
If a finite $\pac$-$\pef$-free poset $P$ is disconnected, then $P = P' \sqcup P''$ for some non-empty $\pac$-$\pef$-free posets $P'$ and $P''$. 
The induction hypothesis shows that $P'$ and $P''$ are in $\Cb$, and thus so is $P$. 
So let us suppose that $P$ is connected.
Because of Lemma~\ref{lem:unique-max-min}, together with the $\pef$-freeness, we may assume without loss of generality 
that $P$ is of the form $[1] \wP P'$ for some poset $P'$. Since $P'$ is a convex subposet of $P$, 
it is $\pac$-$\pef$-free, from which the induction hypothesis implies that $P'$ is in $\Cb$, and thus so is $P$.

\end{proof}

\section{Open problems}\label{sec:conclusion}

\subsection{Notes on the main conjectures}

Resolving any one of Conjectures~\ref{con:chromatic_trees}, \ref{con:poset_trees} and~\ref{con:labeled_rooted_trees} would represent a significant advance.  Taking them in turn, we offer some further observations.

\begin{itemize}
\item Even though our starting point was Conjecture~\ref{con:chromatic_trees}, the only approach taken here is to consider it in terms of Conjecture~\ref{con:poset_trees}.  Perhaps there is a more direct way to tackle the former.  

\item One difficulty we failed to surmount in tackling Conjecture~\ref{con:poset_trees} was how to use the fact that the posets under consideration are trees; any methods used cannot apply to the posets in~Figure~\ref{fig:equal_kp}.

\item A special case of Conjecture~\ref{con:labeled_rooted_trees} worthy of consideration is that of binary trees, defined here as rooted trees where every element has exactly 0 or 2 children.  

Referring to Subsection~\ref{sub:classes}, one could consider series-parallel posets where we allow a mixture of strict and weak edges.   In particular, one could define \emph{fair series-parallel posets} by replacing all appearances of ``$[1]$'' in parts (c) and (d) of Definition~\ref{def:C} with ``$(Q, \om')$.''  Does the analogue of Theorem~\ref{thm:fair-trees} hold?

For any class of posets, labeled or not, we expect the crux of a proof would be the irreducibility, as in Proposition~\ref{prop:irred-C}, \cite[Lemma~3.13]{HaTs17}, and \cite[Theorem~4.19]{LiWe21}.  The general question of the irreducibility of $\kpw$ appears as \cite[Questions~7.2 and~7.3]{McWa14}. The irreducibility of $M_\alpha$ and $F_\alpha$ is shown in \cite{LaPy08}.
\end{itemize}

\subsection{Quasisymmetric power sum bases}

We have given much less consideration to other open problems which we next describe.  Liu and Weselcouch's impressive progress in the naturally labeled case \cite{LiWe21} depends on the expansion of $K_P$ in the (unnormalized) power sum basis $\psi_\alpha$ of type I and, in particular, the combinatorial interpretation of this expansion due to Alexandersson and Sulzgruber \cite{AlSu21}.  Can other bases of \qsym\ give new insight?  Even though the expansion of $K_P$ in other bases might be neither integral nor positive, there could still be interesting combinatorics with appropriate normalization and interpretation of the signs.  We did a careful study of the expansion of $K_P$ in the (unnormalized) power sum basis $\phi_\alpha$ of type II.  Here, we are following the notation of Ballantine et al.~\cite{BDHMN20}, where extensive information about the bases $\psi_\alpha$ and $\phi_\alpha$ can be found. In particular, they show that the power sum symmetric function $p_\lambda$ expands as
\begin{equation}\label{equ:p_lambda}
p_\lambda = \sum_{\alpha\,:\,\widetilde{\alpha} = \lambda} z_\lambda \psi_\alpha
= \sum_{\alpha\,:\,\widetilde{\alpha} = \lambda} z_\lambda \phi_\alpha
\end{equation}
where both sums are over all compositions $\alpha$ whose weakly decreasing reordering is $\lambda$. As usual, $z_\lambda = 1^{m_1}m_1!\,2^{m_2}m_2! \cdots k^{m_k}m_k!$ where $m_i$ is the multiplicity of $i$ in $\lambda$ and where $\lambda_1 = k$. It follows immediately from~\eqref{equ:p_lambda} that when $f$ is a symmetric function, the coefficient of $\psi_\alpha$ in $f$'s $\psi$-expansion equals the coefficient of $\phi_\alpha$ in $f$'s $\phi$-expansion.  We offer the following question: is the converse true?  In other words, does the coefficients being equal give a characterization of symmetric functions? We have only anecdotal evidence related to this question and have not given it significant thought.

\subsection{Principal Specialization}\label{sub:ps}

Given a quasisymmetric function $f(\x)$ in infinitely many variables $x_1, x_2, \ldots$, we denote by $f(1, q, q^2, \ldots, q^{k-1})$ the element of $\mathbb{Z}[q]$ that results from setting $x_i=q^{i-1}$ for all $i \leq k$ and $x_i=0$ otherwise.  This is known as the \emph{principal specialization of order $k$} of $f$ \cite[Sections 7.8, 7.19]{Sta99}.  This specialization has a nice interpretation for $\kpw$: if 
\[
\kpw(1, q, q^2, \ldots, q^{k-1}) = \sum_{N\geq 0} a(N) q^N,
\]
then we see that $a(N)$ counts the number of $(P,\om)$-partitions $f:P \to \{0,\ldots,k-1\}$ of $N$.  Notice that $f$ is now allowed to map to 0 but cannot map to integers larger than $k-1$, and the sum of $f(p)$ as $p$ ranges over $P$ is $N$.  Similar interpretations apply to the principal specializations of $X_G$ and $X_{\diG}$\,.  If we let $k \to \infty$, we get the so-called \emph{stable principal specialization}  $\kpw(1, q, q^2, \ldots)$ that is essentially equivalent to $G_{P,\om}$ appearing in \cite[Subsection~3.15.2]{Sta12}.

Recall that Stanley's original conjecture states that $X_G(\x)$ distinguishes trees.  Clearly $X_G(1,q, q^2, \ldots, q^{k-1})$ contains far less information than $X_G(\x)$, making the following recent conjecture of Loehr and Warrington surprising.

\begin{conjecture}[\cite{LoWa22+}]
$X_G(1,q, q^2, \ldots, q^{n-1})$ distinguishes trees with at most $n$ vertices.  
\end{conjecture}

Certainly an affirmative answer to Loehr and Warrington's conjecture would prove Stanley's conjecture.

Inspired by this, we offer the following conjecture which should be compared with Conjecture~\ref{con:poset_trees}.

\begin{conjecture}\label{con:sp_poset_trees}
$\KK{P}(1, q, q^2, \ldots, q^{n-1})$ distinguishes posets with $n$ elements that are trees. 
\end{conjecture}

We have verified this conjecture for all such posets with at most 10 elements.  In fact, it seems that $\KK{P}(1, q, q^2, \ldots, q^{n-2})$ suffices when the posets have $n$ elements. 

\begin{conjecture}\label{con:sp_poset_trees_one_fewer}
$\KK{P}(1, q, q^2, \ldots, q^{n-2})$ distinguishes posets with $n$ elements that are trees. 
\end{conjecture}

It is not obvious to us whether Conjecture~\ref{con:sp_poset_trees_one_fewer} being true would imply Conjecture~\ref{con:sp_poset_trees}.

\begin{example}
Referring to Figure~\ref{fig:f_support}, the labelel poset $P$ on the left satisfies
\[
\KK{P}(1, q, q^2, q^3) = q^4 + 2q^5 + 4q^6 + 4q^7 + 5q^8 + 4q^9 + 2q^{10} + q^{11}
\]
whereas its dual $P^*$ on the right satisfies
\[
\KK{P^*}(1, q, q^2, q^3) = q^4 + 2q^5 + 4q^6 + 5q^7 + 4q^8 + 4q^9 + 2q^{10} + q^{11}.
\]
So, even though their $F$-supports do not distinguish them, their principal specialization of order 4 does.  

This example also exhibits how the principal specialization of $\KK{P}$ of order $k$ compares to that of its dual $P^*$.  One can check that for general $P$ with $n$ elements, the coefficient of $q^N$ in $\KK{P}(1, q, q^2, \ldots, q^{k-1})$ equals the coefficient of $q^{n(k-1)-N}$ in $\KK{P^*}(1, q, q^2, \ldots, q^{k-1})$.  
\end{example}

However, $\KK{P}(1, q, q^2, \ldots, q^{k-1})$ for $k-1<n-2$ does not distinguish posets with $n$ elements that are trees because it will evaluate to 0 for any posets containing a chain of length $n-2$.  This begs the question of what happens if we consider posets with all weak edges.  Unlike in the case of $\KK{P}(\x)$ where $\KK{P}(\x)$ can be obtained from $K_P(\x)$ and vice versa, the same is not true for their principal specializations.  For example, $K_P(1,q,q^2)$ suffices to distinguish posets with 7 elements that are trees.  It is not clear for general $n$ which values of $k$ are sufficient for $K_P(1,q,\ldots,q^{k-1})$  to distinguish posets with $n$ elements that are trees.  Our computations are consistent with Conjectures~\ref{con:sp_poset_trees} and~\ref{con:sp_poset_trees_one_fewer} remaining true when $\KK{P}$ is replaced by $K_P$.

Finally we note that the principal specialization version of Conjecture~\ref{con:labeled_rooted_trees} is false: $F_{131}$ and $F_{212}$ correspond to $\kpw(\x)$ for different chains with 5 elements but 
\[
F_{131}(1,q,q^2,q^3,q^4) = F_{212}(1,q,q^2,q^3,q^4).  
\]

\subsection{Another quasisymmetric version of Stanley's conjecture}

We close with another question which is wide open and again brings us all the way back to Stanley's original consideration of $X_G(\x)$ distinguishing (undirected) graphs and Conjecture~\ref{con:chromatic_trees}.  We ask the following ill-defined question: does $X_{\diG}(\x, t)$ distinguish \emph{undirected} trees? Here is one concrete way to make this question make sense.

\begin{question}
Given an undirected tree $G$, construct the multiset $\{X_{\diG} (\mathbf{x}, t)\}_{\diG}$ as $\diG$ varies over all orientations of $G$. Do the same for a different undirected tree $H$. Are the multisets for $G$ and $H$ always different?
\end{question}

\section*{Acknowledgements} 

We thank the anonymous referee for suggestions that helped clarify some explanations.  Much of this paper was written while the third author was on sabbatical at Universit\'e de Bordeaux; he thanks LaBRI for its hospitality.  Computations were performed using SageMath \cite{sage}.

\bibliography{distinguishing_trees}
\bibliographystyle{alpha}

\end{document}